\documentclass[10pt, a4paper]{article}
\usepackage{amssymb}
\usepackage{amsmath}
\usepackage{amsfonts}
\usepackage{amscd}
\usepackage{geometry}
\usepackage[all,cmtip]{xy}
\usepackage{makeidx}
\usepackage{titlesec}
\usepackage{fancyhdr}

\newcommand{\C}{\mathbb{C}}
\newcommand{\R}{\mathbb{R}}
\newcommand{\Q}{\mathbb{Q}}
\newcommand{\N}{\mathbb{N}}
\newcommand{\Z}{\mathbb{Z}}
\newcommand{\A}{\mathbb{A}}
\newcommand{\F}{\mathbb{F}}

\DeclareMathOperator{\Symm}{Symm}
\DeclareMathOperator{\HT}{HT}
\DeclareMathOperator{\WD}{WD}

\DeclareMathOperator{\Res}{Res}

\DeclareMathOperator{\rec}{rec}

\DeclareMathOperator{\diag}{diag}

\DeclareMathOperator{\Ind}{Ind}

\DeclareMathOperator{\Frob}{Frob}
\DeclareMathOperator{\Tr}{Tr}
\DeclareMathOperator{\Gal}{Gal}
\DeclareMathOperator{\GL}{GL}

\DeclareMathOperator{\GSp}{GSp}
\DeclareMathOperator{\PGL}{PGL}
\DeclareMathOperator{\PSL}{PSL}
\DeclareMathOperator{\PGSp}{PGSp}
\DeclareMathOperator{\PSp}{PSp}
\DeclareMathOperator{\Sp}{Sp}
\DeclareMathOperator{\SO}{SO}
\DeclareMathOperator{\Sz}{Sz}

\titleformat{\section}[hang]
{\normalfont\filright\large}{\thesection. }{0pt}
{\upshape\bfseries}

\titleformat{\subsection}[hang]
{\itshape}{\thesubsection \ - }{0pt}
{}

\usepackage{amsthm}
\theoremstyle{plain}
\newtheorem{theo}{Theorem}[section]
\newtheorem{prop}[theo]{Proposition}
\newtheorem{lemm}[theo]{Lemma}

\theoremstyle{remark}
\newtheorem{rema}[theo]{\sc Remark}
\theoremstyle{definition}
\newtheorem{defi}[theo]{Definition}

\title{On the images of the Galois representations attached to generic automorphic representations of $\GSp(4)$}

\author{\small LUIS DIEULEFAIT \footnote{Departament d'Algebra i Geometria, Facultat de Matem$\grave{\mbox{a}}$tiques, Universitat de Barcelona, Gran Via de les Corts Catalanes, 585, 08007 Barcelona, Spain, \texttt{ldieulefait@ub.edu}}, ADRI$\acute{\mbox{A}}$N ZENTENO \footnote{Instituto de Matem\'aticas, (Unidad Cuernavaca) UNAM. Av. Universidad s/n. Col. Lomas de Chamilpa CP 62210, Cuernavaca, Mexico. \texttt{matematicazg@ciencias.unam.mx}}}

\date{\today}
\begin{document}

\maketitle


\begin{abstract}
By making use of Langlands functoriality between $\GSp(4)$ and $\GL(4)$, we show that the images of the Galois representations attached to ``genuine" globally generic automorphic representations of $\GSp(4)$ are  ``large"  for a set of primes of density one. Moreover, by using the notion of $(n,p)$-groups (introduced by Khare, Larsen and Savin) and generic Langlands functoriality from $\SO(5)$ to $\GL(4)$ we construct automorphic representations of $\GSp(4)$ such that the compatible system attached to them has large image for all primes.  

\end{abstract}


\section*{Introduction}

Let $f$ be a newform of level $N$, weight $k>1$, and Nebentypus $\chi$ with $q$-expansion $\sum_{n \geq 1} a_n q^n$ ($q = q(z) = e^{2\pi i z}$); and $E := \Q ( a_n: (n,N)=1 )$ be the coefficient field of $f$ which is a number field. By a construction of Shimura and Deligne \cite{De71}, for each maximal ideal $\lambda$ of $\mathcal{O}_E$ (the ring of integers of $E$), one can attach to $f$ a $2$-dimensional Galois representation
\[
\rho_{\lambda}: G_\Q \rightarrow  \GL(2,\overline{E}_\lambda)
\]  
unramified at all rational primes $p \nmid N \ell$ (where $E_\lambda$ denotes the completion of $E$ at $\lambda$ and $\ell$ denotes the rational prime below $\lambda$) and such that $\Tr(\rho_{\lambda}(\Frob_p)) = a_p$ and $\det (\rho_\lambda (\Frob_p)) = \chi(p) p^{k-1}$ for every rational prime $p \nmid N \ell$. 

Let $\overline{\rho}_{\lambda}$ be the semisimplification of the reduction of $\rho_\lambda$ modulo $\lambda$ and $\overline{\rho}^{proj}_\lambda$ its projectivization. We say that $f$ is \emph{exceptional} at $\lambda$ if the image of $\overline{\rho}_\lambda^{proj}$ is neither $\PSL(2,\F_{\ell^s})$ nor $\PGL(2,\F_{\ell^s})$ for some integer $s>0$. In the 70's and 80's Carayol, Deligne, Langlands, Momose, Ribet, Serre and Swinnerton-Dyer proved that, if $f$ does not have complex multiplication then $f$ is exceptional at most at finitely many $\lambda$ (see the introduction of \cite{Rib85}).

In this paper, we prove a weak version of the analogous result for automorphic representations of $\GSp(4,\A_\Q)$. More precisely, let $\pi =\pi_\infty \otimes \pi_f$ be a unitary cuspidal irreducible automorphic representation of $\GSp(4,\A_\Q)$ of cohomological weight $(m_1,m_2)$, $m_1 \geq m_2 \geq 0$, for which $\pi_\infty$ belongs to the discrete series.  Thanks to the work of Laumon, Taylor and Weissauer \cite{we05} we can attach to $\pi$ a number field $E$, and for all maximal ideal $\lambda$ of $\mathcal{O}_E$, a 4-dimensional Galois representation  
\[
\rho_{\lambda}: G_\Q \rightarrow  \GL(4, \overline{E}_\lambda)
\]
which is unramified outside $S \cup \{ \ell \}$, where $S$ is the set of places where $\pi$ is not spherical. Then, by using the generic Langlands transfer from $\GSp(4)$ to $\GL(4)$ and some recent results about residual irreducibility of compatible systems, we prove that if $\pi$ is genuine and globally generic then $\pi$ is ``exceptional" at most at a set of primes of density zero. 

Alternatively, a way to build representations with large image is by using the notion of $(n,p)$-groups introduced in \cite{KLS}. With this tool we prove that a symplectic compatible system of $4$-dimensional semisimple Galois representations attached to a RAESDC automorphic representation of $\GL(4)$ has only finitely many exceptional primes if it is maximally induced in an appropriate prime (see Section \ref{se3}).

In the same direction, stronger results have been proved in \cite{DW} and \cite{DZ} for compatible systems attached to classical and Hilbert modular forms. More precisely, in loc. cit. there have been constructed (Hilbert) modular forms such that the compatible systems attached to them, do not have exceptional primes.
At the end of this paper, by making use of Langlands functoriality (from $\SO(5)$ to $\GL(4)$ and from $\GL(4)$ to $\GSp(4)$), we generalize this construction to automorphic representations of $\GSp(4,\A_\Q)$.


\subsection*{Notation}
We shall use the following notations. Let $K$ be a number field or a $\ell$-adic field. We denote by
$\mathcal{O}_K$ its ring of integers. For a maximal ideal $\mathfrak{q}$ of $\mathcal{O}_K$, we let $D_\mathfrak{q}$ and $I_\mathfrak{q}$ be the corresponding decomposition and inertia group at $\mathfrak{q}$, respectively. We denote by $G_K$ the absolute Galois group of $K$. In particular if $K$ is a number field by a prime of $K$ we mean a nonzero prime ideal of $\mathcal{O}_K$. 
In this paper, the symplectic similitude group $\GSp(4)$ is defined with respect to the following skew-symmetric matrix:
\[
\left( \begin{array}{cc}
0 & J \\
-J & 0 \end{array} \right),
\mbox{ where }
J = \left( \begin{array}{cc}
0 & 1 \\
1 & 0 \end{array} \right).
\]
Finally, $\WD(\rho)^{F-ss}$ will denote the Frobenius semisimplification of the Weil-Deligne representation attached to a representation $\rho$ of $G_{\Q_\ell}$, and $\rec$ is the notation for the Local Langlands Correspondence, which attaches to an irreducible admissible representation of $\GL(4,\Q_\ell)$ (resp. $\GSp(4,\Q_\ell)$) a Weil-Deligne representation of the Weil group $W_{\Q_\ell}$ as in \cite{HT01} (resp. \cite{GT11}). 


\subsection*{Acknowledgments}
This paper is part of the second author's PhD thesis. Part of this work has been written during a stay at  Institut Galilee of the University of Paris 13 and a stay, of the second author, at the Mathematical Institute of the University of Barcelona. The authors would like to thank these institutions for their
support and the optimal working conditions. The second author wants to thank Jacques Tilouine for many stimulating conversations, and Sara Arias de Reyna and Michael Larsen for a useful correspondence. 
Finally, the authors want to give special thanks to the anonymous referee, whose comments and suggestions have greatly improved the presentation and readability of this paper.
The research of L.D. was supported by an ICREA Academia research prize. The research of A.Z. was supported by LAISLA and by the CONACYT grant no. 432521/286915.


\section{RAESDC automorphic representations of $\GL(4)$} 
In this section we review some facts about RAESDC automorphic representations and the Galois representations associated to them. Our main references are \cite[Section 2]{CT}, \cite{CH13} and \cite{Clo}.
Let $\A$ be the ring of rational adeles. By a RAESDC (regular algebraic, essentially self-dual, cuspidal) automorphic representation of $\GL(4, \A)$ we mean a pair $(\Pi, \mu)$ consisting of a cuspidal automorphic representation $\Pi = \Pi_\infty \otimes \Pi_f$ of $\GL(4, \A)$ and an algebraic character $\mu: \A^\times / \Q^\times \rightarrow \C^\times$ such that:
\begin{enumerate}
\item (essentially self-dual) $\Pi \cong \Pi ^{\vee} \otimes \mu$.
\item (regular algebraic) $\rec_\R (\Pi_\infty \otimes \vert \; \vert ^{-\frac{3}{2}})$ restricted to $\C^\times$ is a direct sum of pairwise distinct algebraic characters. Here $\rec_\R$ denotes the Local Langlands correspondence for $\GL(4,\R)$  with the Langlands'normalization.
\end{enumerate}

Recall that, by the Lemma of Purity \cite[Lemma 4.9]{Clo}, there exist an integer $w \in \Z$ such that  $\rec_\R (\Pi_\infty)$ restricted to $\C^\times$ is of the form
\[
z \mapsto \diag \left( z^{\alpha_1} \overline{z}^{w - \alpha_1}, \ldots,  z^{\alpha_4} \overline{z}^{w-\alpha_4} \right),
\]
where $ \alpha_{i} \in 3/2 + \Z$ and $\alpha_1 > \cdots > \alpha_4$. The tuple $\alpha = (\alpha_1, \cdots, \alpha_4)$ are called the \emph{infinite type} of $\Pi$. We remark that $w=0$ if and only if $\Pi$ is unitary. Finally we define the \emph{weight} of $\Pi$ as the tuple $a = (a_1, \cdots, a_4)$ where $a_i$ is defined by the formula $a_i = i - \alpha_{5-i} - 5/2$, so $a_1 \geq \cdots \geq a_4$.  

Let $n \in \N$. A \emph{compatible system} $\rho = (\rho_{\lambda})_{\lambda}$ of $n$-dimensional Galois representations of $G_\Q$ consist of the following data:
\begin{enumerate}
\item A number field $E$.
\item A finite set $S$ of primes of $\Q$.
\item For each prime $p \notin S$, a monic polynomial $P_p(X) \in \mathcal{O}_E[X]$.
\item For each prime $\lambda$ of $E$ (together with fixed embeddings $E \hookrightarrow E_\lambda \hookrightarrow \overline{E}_\lambda$) a continuous Galois representation
\[
\rho_\lambda : G_\Q \rightarrow \GL(n,\overline{E}_\lambda)
\]
such that $\rho_\lambda$ is unramified outside $S \cup \{ \ell \}$ (where $\ell$ is the residual characteristic of $\lambda$) and such that for all $p \notin S \cup \{ \ell \}$ the characteristic polynomial of $\rho_\lambda(\Frob_p)$ is equal to $P_p(X)$.
\end{enumerate} 

\begin{theo}\label{raesdc}
Let $(\Pi, \mu)$ be a RAESDC automorphic representation of $\GL(4, \A)$. Then, there exist a number field $E$, a finite set $S$ of primes of $\Q$, and compatible systems of semisimple Galois representations
\[
\rho_{\Pi, \lambda} : G_{\Q} \rightarrow \GL(4, \overline{E}_\lambda) \; \quad \mbox{ and } \; \quad
\rho_{\mu, \lambda} : G_\Q \rightarrow \overline{E}^{\times}_\lambda
\]
where $\lambda$ ranges over all primes of $E$ (together with fixed embeddings $E \hookrightarrow E_{\lambda} \hookrightarrow \overline{E}_{\lambda}$) such that the following properties are satisfied.
\begin{enumerate}
\item $\rho_{\Pi,\lambda} \cong \rho_{\Pi,\lambda} \otimes \chi^{1-n}_\ell\rho_{\Pi,\lambda}(\mu)$, where $\chi_\ell$ denotes the $\ell$-adic cyclotomic character.  
\item The representations $\rho_{\Pi,\lambda}$ and $\rho_{\mu,\lambda}$ are unramified outside $S \cup \{ \ell \}$.
\item The representations $\rho_{\Pi,\lambda} \vert _{G_{\Q_\ell}}$ and $\rho_{\mu,\lambda} \vert _{G_{\Q_\ell}}$ are de Rham, and if $\ell \notin S$, they are crystalline.
\item $\rho_{\Pi,\lambda}$ is regular, and the set of Hodge-Tate weights $\HT(\rho_{\Pi,\lambda})$ is equal to:
\[
\{ a_4, a_3+1, a_2+2, a_1+3 \}.
\]
\item Fix any isomorphism $ \iota : \overline{E}_{\lambda} \simeq \C$ compatible with the inclusion $E \subset \C$. Then
\[
\iota \WD(\rho_{\Pi,\lambda} \vert_{\Q_p})^{F-ss} \cong \rec (\Pi_p \otimes \vert \det \vert _p ^{-3/2}).
\]
\end{enumerate}
\end{theo}
\begin{proof}
This theorem follows from the analogous result for RACSDC (regular algebraic, conjugate self-dual, cuspidal) automorphic representations over CM fields, by using the solvable base change theorems of Arthur-Clozel \cite{AC89} and the patching lemma of \cite{Sor}.
The proof of the existence of the representations $\rho_{\Pi, \lambda}$, in the RACSDC case, can be found in \cite{CH13} and the strong form of local-global compatibility is proved in \cite{Car12a} and \cite{Ca12}. 
\end{proof}


\section{Cuspidal automorphic representations of $\GSp(4)$} \label{lasec}

Let $\pi = \pi_\infty \otimes \pi_f$ be a globally generic, irreducible, cuspidal, automorphic representation of $\GSp(4,\A)$ with cohomological weight $(m_1,m_2)$, $m_1 \geq m_2 \geq 0 $, and central character $\omega_{\pi}$, such that $\pi_\infty$ belongs to the discrete series. Let $w=m_1 +m_2$ and $\pi^\circ := \pi \otimes \vert c \vert^{w/2}$, where $c$ denotes the similitude character of $\GSp(4)$. From now on, we will assume that $\pi^\circ$ is unitary and that $\pi$ is neither CAP nor endoscopic. 

\begin{rema}
When $\pi$ is CAP or endoscopic, it is well known that the Galois representations associated to $\pi$ are reducible, so they cannot have large image. See Section 3.2 of \cite{mok} and the Introduction of \cite{we05} for more details. 
\end{rema}

With $\pi$ satisfying the above hypotheses, we can lift $\pi$ to a cuspidal automorphic representation $\Pi$ of $\GL(4,\A)$  with central character $\omega_\Pi = \omega_\pi^2$ and such that its archimedean $L$-parameter $\phi$ has the following restriction to $\C^*$:
\[
z \mapsto \vert z \vert ^{-w} \cdot \diag \left( (z/\overline{z})^{\frac{v_1+v_2}{2}} , (z/\overline{z})^{\frac{v_1-v_2}{2}} , (z/\overline{z})^{- \frac{v_1-v_2}{2}} , (z/\overline{z})^{- \frac{v_1+v_2}{2}} \right),
\]
where $v_1 = m_1 +2$ and $v_2 = m_2 + 1$ give the Harish-Chandra parameter of $\pi_\infty$ (see Section 2 of \cite{Sor} or Section 3.1 of \cite{mok}). Such lifting satisfies the following properties:
\begin{enumerate}
\item $\Pi \simeq \Pi^\vee \otimes \omega_\pi$,
\item $L^S(s,\Pi, \wedge ^2 \otimes \omega_\pi^{-1})$ has a pole at $s=1$, and
\item it is a strong lift, that is, $\rec(\pi_v) = \rec(\Pi_v)$ for each place $v$. 
\end{enumerate} 

It has been known for some time that we can obtain a weak lift using theta series. This was first announced by Jacquet, Piatetski-Shapiro and Shalika, but to the best of our knowledge they never wrote up a proof.  
However, there is an alternative proof due to Asgari and Shaidi \cite{AS08} relying on the converse theorem. The strong lift and the characterization of its image is due to Gan and Takeda (see Section 12 of \cite{GT11}).

We say that a compatible system $\rho =(\rho_\lambda)_\lambda$ of $4$-dimensional Galois representations of $G_\Q$ is \emph{symplectic} if for all $\lambda$ the representation $\rho_\lambda$ is of the form $G_\Q \rightarrow \GSp(4, \overline{E}_\lambda)$.

\begin{theo}\label{21}
Suppose that $\pi$ is a cuspidal automorphic representation of $\GSp(4, \A)$ satisfying the hypotheses in the beginning of this section. Let $S$ denote the set of places where $\pi$ is not spherical. 
Then, there exist a number field $E$, and a symplectic compatible system of semisimple Galois representations
\[
\rho_{\pi, \lambda} : G_{\Q} \rightarrow \GSp(4, \overline{E}_\lambda)
\]
where $\lambda$ ranges over the finite places of $E$ (together with fixed embeddings $E \hookrightarrow E_{\lambda} \hookrightarrow \overline{E}_{\lambda}$) such that the following properties are satisfied.
\begin{enumerate}
\item The representation $\rho_{\pi,\lambda}$ is unramified outside $S \cup \{ \ell \}$.
\item The representations $\rho_{\pi,\lambda} \vert _{G_{\Q_\ell}}$ are de Rham, and if $\ell \notin S$, they are crystalline.
\item The set of Hodge-Tate weights $\HT(\rho_{\pi,\lambda})$ is equal to:
\[
\{0, m_2+1, m_1+2, m_1+m_2+3 \}.
\]
\item Fix any isomorphism $ \iota : \overline{E}_{\lambda} \simeq \C$ compatible with the inclusion $E \subset \C$. Then
\[
\iota \WD(\rho_{\pi,\lambda} \vert_{\Q_p})^{F-ss} \cong \rec (\pi_p \otimes \vert c \vert _p ^{-3/2}).
\]
\end{enumerate}
\end{theo}
\begin{proof}
First, from the previous discussion we can lift $\pi$ to a RAESDC automorphic representations $(\Pi,\mu)$ of $\GL(4,\A)$. Then we define the compatible system of Galois representations $\rho_{\pi, \lambda}$ associated to $\pi$ as $\rho_{\Pi, \lambda}$ (the compatible system of Galois representations associated to $(\Pi,\mu)$ in Theorem \ref{raesdc}). In fact, $\rho_{\pi, \lambda}$ can be also constructed directly from the cohomology of a suitable Siegel threefold (see Theorem I of \cite{we05}). 
As the representations of both constructions have the same Frobenius traces, we can conclude, by the Brauer-Nesbitt theorem, that the representations obtained in both constructions are isomorphic. 

On the other hand, we know that all globally generic cuspidal automorphic representation of $\GSp(4,\A)$ have multiplicity one (see \cite{JS07}).  Then, from Theorem IV of \cite{we05}, $\rho_{\pi,\lambda}$ takes values in $\GSp(4,\overline{E}_{\lambda})$. The Hodge-Tate weights, item $iii)$, can be calculated as at the end of Section 4.4 of \cite{Sor}, and the local-global compatibility, item $iv)$, is as in Theorem 3.1 of \cite{mok}. 
\end{proof}

\begin{rema} Note that from the property $iv)$ of the previous theorem follows that the conductor of $\rho_{\pi,\lambda}$ is independent of $\lambda$. Then, this can be called the \emph{conductor} of the compatible system.
\end{rema}


\section{Galois Representations with large images}\label{se4}

Let $\rho_{\pi,\lambda}: G_{\Q} \rightarrow \GSp(4,\overline{E}_\lambda)$ be a $4$-dimensional symplectic Galois representations as in Theorem \ref{21}.   
In this case, we can take as $E$, the number field generated over $\Q$ by the coefficients of the characteristic polynomials of all $\rho_{\pi,\lambda}(\Frob_{p})$, $p \notin S$. By using Lemma 3 of \cite{DKR}, we can define the residual mod $\lambda$ Galois representation $\overline{\rho}_{\pi,\lambda}:G_\Q \rightarrow \GSp(4, \F_\lambda)$, where $\F_\lambda = \mathcal{O}_E / \lambda$. 
We denote by $\overline{\rho}_{\pi,\lambda}^{proj}$ the composition of $\overline{\rho}_{\pi,\lambda}$ with the natural projection $\GSp(4,\F_\lambda) \rightarrow \PGSp(4,\F_\lambda)$. 

Let $\pi$ be a cuspidal automorphic representation of $\GSp(4,\A)$ as in Theorem \ref{21}, we say that $\pi$ is \emph{exceptional} at a prime $\lambda$ if the image of $\overline{\rho}_{\pi,\lambda}^{proj}$ is neither $\PSp(4,\F_{\ell^s})$ nor $\PGSp(4,\F_{\ell^s})$ for some integer $s > 0$. 
On the other hand, such $\pi$ will be called \emph{genuine}, if it is neither a symmetric cube lift from $\GL(2)$, nor an automorphic induction after lift to $\GL(4)$. The rest of this section is devoted to prove the following result.

\begin{theo}\label{prin}
Let $\pi$ be a genuine cuspidal automorphic representation of $\GSp(4,\A)$ satisfying the hypotheses in the beginning of the Section \ref{lasec}. Then $\pi$ is exceptional at most at a set of primes of density zero.
\end{theo}

The proof is inspired by \cite{Di02} where the case of genuine cuspidal automorphic representations of $\GSp(4,\A)$ of level $1$ and parallel weight was proved.
As in Dieulefait's paper, the proof is done by considering all possible images of $\overline{\rho}^{proj}_{\pi, \lambda}$, given by the classification of maximal subgroups of $\GSp(4, \F_{\ell^r})$. Such classification was first provided by Mitchell in \cite{mi14}. However, we use a more modern formulation due to Aschbacher \cite{Asc84} which is as follow.

\begin{theo}\label{clasific}
Let $\ell$ be a odd rational prime and $r$ be an integer $>0$. Let $G$ be a maximal subgroups of $\GSp(4,\F_{\ell^r})$ which does not contain $\Sp(4,\F_{\ell^r})$. Then at least one of the following holds: 
\begin{enumerate}
\item $G$ stabilizes a totally singular or a non-singular subspace;
\item $G$ stabilizes a decomposition $\F_{\ell^r}^4 = V_1 \oplus V_2$, $\dim(V_i) = 2$;
\item $G$ stabilizes a structure of $\F_{\ell^{2r}}$-vector space on $\F^4_{\ell^r}$; 
\item $G$ is a cross characteristic group of order smaller than $5040$;
\item the projectivization of $G$ is an almost simple group isomorphic to $\PGL(2,\F_{\ell^r})$;
\item the projectivization of $G$ is an almost simple group isomorphic to $\PSp(4,\F_{\ell^s})$ or $\PGSp(4,\F_{\ell^s})$, for some integer $s>0$ dividing $r$.
\end{enumerate}
\end{theo}

For more details and relevant definitions see Chapter 2 and Chapter 4 of \cite{BHRD13}. The Aschbacher's classification is in fact a much more general result, which give a classification of maximal subgroups of all the finite classical groups.   

When $\ell-1 > m_1 + m_2 + 3$ and $\ell \notin S$, we have from Theorem \ref{21} that $\rho_{\pi,\lambda}$ is crystalline with Hodge-Tate weights $\{ 0, m_2+1, m_1+2, m_1+m_2 +3 \}$. Then we have the following result which follows from Fontaine-Laffaille theory \cite[Theorem 5.3]{FL82} (see also Section 3.1 of \cite{Ur}).

\begin{prop}\label{iner}
Let $\pi$ be a cuspidal automorphic representation of $\GSp(4, \A)$ as in Theorem \ref{21}. Then for all prime $\ell \notin S$, such that $\ell-1 > m_1 + m_2 + 3$, we have the following possibilities for the action of the inertia group at $\ell$:
\begin{equation*}
\begin{split}
\overline{\rho}^{ss}_{\pi, \lambda} \vert _{I_{\ell}} \simeq &  \diag(\chi^{m_1+m_2+3},\chi^{m_1+2}, \chi^{m_2+1}, 1),\\
 & \diag(\chi^{m_1+2}, \chi^{m_2+ 1}, \psi^{(m_1+m_2+3)\ell}, \psi^{m_1+m_2+3}), \\
 & \diag( \psi^{(m_1+ 2)+(m_2 + 1) \ell}, \psi^{(m_2+1) + (m_1 + 2)\ell}, \chi^{m_1+m_2+3}, 1 ), \\
  & \diag(\psi^{(m_1+2) + (m_2 + 1)\ell}, \psi^{(m_2+1) + (m_1+2)\ell}, \psi^{(m_1+m_2 + 3)\ell}, \psi^{m_1+m_2+3}),
\end{split}
\end{equation*}
where $\chi$ is the cyclotomic character and $\psi$ denotes a fundamental character of level 2. 
\end{prop}

We remark that the particular choice of the exponents and the fundamental characters in the previous result are deduced from the fact that the roots of the characteristic polynomial of an element of $\GSp(4)$ come in two pairs having the same product.

Now, we are ready to give the proof of Theorem \ref{prin}, which will be given by considering the following cases:


\subsection{Reducible images}

In this section, we will deal with the reducible cases. We remark that, instead of following Dieulefait's proof (which depends on the generalized Ramanujan's conjecture and Serre's conjecture), we use the recent results of $\cite{BLGGT14}$ and $\cite{CG14}$ about reducibility of compatible systems. 

First, recall that in the proof of Theorem \ref{21} we define the compatible system of Galois representations $\rho_{\pi, \lambda}$ associated to $\pi$ as the compatible system of Galois representation $\rho_{\Pi, \lambda}$  associated to a RAESDC automoprihc representation $\Pi$ of $\GL_4(\A)$ via the generic Langlands transfer from $\GSp(4)$ to $\GL(4)$.
By Theorem 3.2 of \cite{CG14}, we have that the representations $\rho_{\Pi, \lambda}$ on the compatible system attached to a RAESDC automorphic representation $\Pi$ of $\GL(4, \A)$ are absolutely irreducible for a set of primes $\lambda$ of density one. Then, by Proposition 5.3.2 of \cite{BLGGT14}, we have that the residual representations $\overline{\rho}_{\Pi,\lambda}$ are irreducible for a set of primes $\lambda$ of density one. Thus, we can conclude that the reducible cases can only happen for a set of primes of density zero. 

\begin{rema}\label{ramifi}
As the results of $\cite{BLGGT14}$ and $\cite{CG14}$ (see also $\cite{CG}$) only works for a set of density one of primes, we are not able to prove in general that $\pi$ is non exceptional for all but finitely many primes $\lambda$. 
However partial results are well known. For example, in \cite{Di02} the first author proved that if $\pi$ is of parallel weight $k$ and such that $\pi_p$ is spherical for all prime $p$ (i.e. $S = \emptyset$) then $\overline{\rho}_{\pi,\lambda}$ is irreducible for all but finitely many primes $\lambda$. 
In fact, the method of loc. cit. works when $\pi$ is of parallel weight $k$ and the Galois representations associated to it are semistable at primes in $S$ (see Theorem 2.3 of \cite{Di07}). 

Recently, in \cite[Section 5]{We}, has been proved that the representations $\overline{\rho}_{\pi,\lambda}$ are irreducible for all but finitely many primes $\lambda$ without conditions on $S$. Then a strong version of our result is now available.
\end{rema}


\subsection{Image equal to a group having a reducible index two subgroup}
Assume that we are in the case $ii)$ or $iii)$ of Theorem \ref{clasific}. In these cases $\overline{\rho}_{\pi,\lambda}$ is the induction of some 2-dimensional representation $\sigma_\lambda$ of $G_L$ that is not the restriction of a $2$-dimensional representation of $G_\Q$, for $L$ a quadratic extension of $\Q$. 
Now, assume that for infinitely many primes $\lambda$
\[
\rho_{\pi,\lambda} \equiv \Ind_L^{\Q}(\sigma_\lambda)\mod \lambda. 
\]
A priori $L$ and $\sigma_\lambda$ depend on the prime $\lambda$. By using the description of the image of inertia at $\ell$ given in Proposition $\ref{iner}$, we have that $L$ is unramified at $\ell$ for $\ell$ sufficiently large and by Dirichlet principle we can assume without loss of generality that $L$ is independent of $\lambda$ (see the arguments in Section 3 of \cite{DN11}). Since this induced representation is irreducible (because the reducible case has been covered before), we have that 
\[
\Tr(\overline{\rho}_{\pi,\lambda} (\Frob_p)) \equiv 0 \mod \lambda
\]
for all $p\notin S \cup \{ \ell \}$ inert in $L$. Since this holds for infinitely many
primes $\lambda$, we have that $\rho_\lambda = \Ind_L^{\Q} (\sigma'_\lambda)$, for some two-dimensional family of $\lambda$-adic representations $\{ \sigma_\lambda ' \}$ of $G_L$, and that
\[
\rho_{\pi,\lambda} = \rho_{\pi,\lambda} \otimes \eta
\]
for all $\lambda$, where $\eta$ is the quadratic character of $L/\Q$ (see Section 4.2 of \cite{DN11}). Then, since $\rho_{\pi,\lambda} = \rho_{\Pi,\lambda}$ for some RAESDC automorphic representation $\Pi$ of $\GL(4, \A)$, by strong multiplicity one for $\GL(4)$ (see \cite{JS81}), we have that
\[
\Pi = \Pi \otimes \eta.
\]
By applying Theorem 4.2 (p. 202) of \cite{AC89}, we deduce that $\Pi$ is an automorphic induction from the quadratic field $L/\Q$. Hence $\pi$ is not genuine. Therefore, these cases of our classification of maximal subgroups of $\GSp(4,\F_{\ell^r})$ can only happen for finitely many primes $\lambda$. 


\subsection{Image equal to the stabilizer of a twisted cubic}

Now, we will deal whit the case $v)$ of Theorem \ref{clasific}.  In this case all matrices are of the form (see page 233 of \cite{Hi85} and Proposition 5.3.6.i of \cite{BHRD13}):
\[
\Symm ^3 \left( \begin{array}{cc}
a & c \\
b & d \end{array} \right)
=
\left( \begin{array}{cccc}
a^3 & a^2c & ac^2 & c^3 \\
3a^2b & a^2d+2abc & bc^2+2acd & 3c^2d \\
3ab^2 & b^2c+2abd & ad^2+2bcd & 3cd^2 \\
b^3 & b^2d & bd^2 & d^3 \end{array} \right),
\]
then
\begin{equation}\label{cubis}
\rho_{\pi,\lambda} \equiv \Symm^3(\sigma_\lambda) \mod \lambda,
\end{equation}
where $\sigma_\lambda$ is a 2-dimensional Galois representation.
Assume that for infinitely many primes $\lambda$ the congruence (\ref{cubis}) is satisfied. If we suppose that $\ell \notin S$ and $\ell -1 > m_1+m_2+3$, comparing the structure of $\Symm^3(\sigma_\lambda)$ with the four possibilities for the image of inertia subgroup at $\ell$ given in the Proposition \ref{iner}, we have that this case can only happen if the weight of $\pi$ is of the form $(2m_2,m_2)$.  We remark that this affirmation is independent of the choice of basis because the eigenvalues of a matrix $\Symm ^3(M)$, $M \in \GL(2)$, is of the form $\alpha^3, \alpha^2 \beta, \alpha \beta^2, \beta^3$ where $\alpha$, $\beta$ are the eigenvalues of $M$ and our comparison just depend on the eigenvalues.
So, in this case we have that the residual mod $\lambda$ representation $\sigma_\lambda$, when restricted to the inertia group at $\ell$, is as follows:
\[
\left( \begin{array}{cc}
\chi^{m_2+1} & * \\
0 & 1 \end{array} \right)
\quad \mbox{ or }\quad
\left( \begin{array}{cc}
\psi ^{(m_2+1)\ell} & 0 \\
0 & \psi^{m_2+1} \end{array} \right).
\]
Then, by Serre's conjecture \cite{Se87} (which is now a theorem, cf. \cite{KW}, \cite{KWb} and \cite{Di12}), for every $\lambda$ that falls in this case, there is a classical cuspidal Hecke eigenform $f_\lambda$ of weight $m_2+2$ and level $N$ such that
\[
\rho_{\pi,\lambda} \equiv \Symm^3(\sigma_{f_\lambda,\lambda}) \mod \lambda,
\]
where $N$ divides the conductor of the compatible system attached to $\pi$. Then we have finitely many possibilities for the modular form $f_\lambda$, and by the Dirichlet principle, we can assume that $f_\lambda =f$ is independent of $\lambda$. Thus we have that 
\[
\rho_{\pi,\lambda} \equiv \Symm^3(\sigma_{f,\lambda}) \mod \lambda,
\] 
for infinitely many $\lambda$, therefore $\rho_{\pi,\lambda} = \Symm^3(\sigma_{f,\lambda})$ for all $\lambda$. Then, as in the previous case, by strong multiplicity one theorem $\pi$ must be the symmetric cube of some cusp form and $\pi$ is not genuine. Hence we can have image the stabilizer of a twisted cubic at most for finitely many primes $\lambda$. 


\subsection{The rest of exceptional images}
Finally, we will deal with the case $iv)$ of Theorem \ref{clasific}. In this case, comparing the exceptional groups $G \subset \Sp(4,\F_{\lambda})$ (its order and structure, see Table 8.12 and Table 8.13 of \cite{BHRD13}) with the fact that the image of $\overline{\rho}_{\pi,\lambda}^{proj}$ contains the image of $\overline{\rho}^{proj}_{\pi, \lambda} \vert _{I_\ell}$ (assuming $\ell \notin S$ and $\ell -1 > m_1 + m_2 + 3$) described in Proposition \ref{iner}, we concluded that this case can only happen for finitely many primes $\lambda$.
  
\subsection{Conclusion}

Having gone through all cases in Theorem \ref{clasific} (except $vi)$) we conclude that, if $\pi$ is genuine satisfying the hypotheses in the beginning of the Section \ref{lasec}, we have at most a set of primes $\lambda$ of density zero where $\pi$ is exceptional.


\section{Maximally induced representations}\label{se3}

It was observed by Khare and Wintenberger \cite{KW} (see also \cite{KLS}) that the existence of exceptional primes in a compatible system can be avoided by imposing certain local conditions on the Galois representations. 
More precisely, let $p,q \geq 5$ be distinct primes such that $p \equiv 1 \mod 4$ and the order of $q$ mod $p$ is $4$. Denote by $\Q_{q^4}$ the unique unramified extension of $\Q_q$ of degree $4$. Recall that $\Q^{\times}_{q^4} \simeq \mu_{q^4-1} \times U_1 \times q^{\Z}$, where $\mu_{q^4-1}$ is the group of $(q^4-1)$-th roots of unity and $U_1$ is the group of 1-units. We consider a character $\chi_q : \Q^{\times}_{q^4} \rightarrow \overline{\Q}^{\times}_\ell$, such that:
\begin{enumerate}
\item $\chi_q$ has order $2p$,
\item $\chi_q (q) = -1$, and
\item $\chi_q \vert _{\mu_{q^4-1} \times U_1}$ is of order $p$.
\end{enumerate}
By local class field theory, we can regard $\chi_q$ as a character (which by abuse of notation we call also $\chi_q$) of $G_{\Q_{q^4}}$ or of $W_{\Q_{q^4}}$.
In \cite{KLS} it is proved that the representation $ \rho_q = \Ind^{G_{\Q_q}}_{G_{\Q_{q^4}}} (\chi_q)$ is irreducible and symplectic, in the sense that it can be conjugated to take values in $\Sp(4, \overline{\Q}_\ell)$.

Let $\ell \neq p$, $\alpha: G_{\Q_q} \rightarrow \overline{\Q}_\ell^\times$ be an unramified character, and $\overline{\chi}_q$ (resp. $\overline{\alpha}$) be the composite of $\chi_q$ (resp. $\alpha$) and the projection $\overline{\Z}_\ell \twoheadrightarrow \overline{\F}_\ell$.
Note that the image of the reduction $\overline{\rho}_q$ of $\rho_q$ in $\GL(4,\overline{\F}_\ell)$ is $\Ind^{G_{\Q_q}}_{G_{\Q_{q^4}}} (\overline{\chi}_q)$ which is an irreducible representation and the representation $ \overline{\rho}_q \otimes \overline{\alpha}$ is irreducible too.

\begin{defi}\label{tm}
Let $p,q,\ell$ be primes and $\chi_q, \alpha$ be characters, all as above. We say that a Galois representation 
\[
\rho_\ell :G_\Q \rightarrow \GSp(4, \overline{\Q}_\ell),
\] 
is \emph{maximally induced} at $q$ of order $p$ if the restriction of $\rho_\ell$ to a decomposition group at $q$ is equivalent to $\Ind^{G_{\Q_q}}_{G_{\Q_{q^4}}} (\chi_q) \otimes \alpha.$
\end{defi} 

\begin{rema}\label{prim}
Let $N \in \N$. Note that, if we choose a prime $p \equiv 1 \mod 4$ greater than $\max\{N, 13\}$, then Chevotarev's density theorem allows us to choose a prime $q \geqslant 5$ (from a set of positive density) which splits completely in $\Q(i,\sqrt{p_1}, \ldots , \sqrt{p_m})$ (where $p_1, \ldots , p_m$ are the prime divisors of $N$) and such that $q^2 \equiv -1 \mod p$.
\end{rema}

\begin{theo}\label{rib2}
Let $N$, $p$, and $q$ as in Remark \ref{prim}. Let $k$ be a positive integer and $\ell \neq p,q$ be a prime such that $\ell > 24k + 1$ and $\ell \nmid N$.
Let $\rho_\ell: G_\Q \rightarrow \GSp(4,\overline{\Q}_\ell)$ be a Galois representation, which ramifies only at the primes dividing $Nq\ell$, and  such that a twist of $\overline{\rho}_\ell$ by some power of the cyclotomic character is regular in the sense of Definition 3.2 of \cite{ADW14} with tame inertia weights at most $k$. 
If $\rho_\ell$ is maximally induced at $q$ of order $p$, then the image of $\overline{\rho}_\ell^{proj}$ is $\PSp(4,\F_{\ell^s})$ or $\PGSp(4,\F_{\ell^s})$ for some integer $s>0$.
\end{theo}

\begin{proof}
We will closely follow the proof of Theorem 1.5 of \cite{ADW14}. As in the previous section we will proceed by cases. 

\subsection{Reducible cases}

Since $\rho_\ell$ is maximally induced at $q$ and $\ell \neq p$, $\overline{\rho}_\ell \vert _{D_q}$ is absolutely irreducible. Hence $\overline{\rho}_\ell$ is absolutely irreducible and the reducible cases in the classification of maximal subgroups of $\GSp(4,\F_{\ell^r})$ cannot happen.


\subsection{Induced cases}\label{indus}

Now suppose that the image of $\overline{\rho}_\ell$ corresponds to an irreducible subgroup inside some of the maximal subgroups in cases $ii)$ or $iii)$ of Theorem \ref{clasific}. 
Because this case is very similar to Lemma 3.7 of \cite{ADW14}, we will omit some details. In these cases there exist a quadratic extension $L \subseteq \Q(i, \sqrt{\ell}, \sqrt{q}, \sqrt{p_1}, \ldots, \sqrt{p_m})$ (where $p_1, \ldots , p_m$ are the prime divisors of $N$) with Galois group $H = \Gal(\overline{\Q}/L) \leq G_\Q$ and a representation $\overline{\sigma}_\ell : H \rightarrow \GL(2, \overline{\F}_\ell)$ such that 
\[
\overline{\rho}_\ell \cong \Ind _H ^{G_\Q} (\overline{\sigma}_\ell). 
\]

Applying Mackey's formula to $\Res _{G_{\Q_q}}^{G_\Q} \left(\Ind^{G_\Q}_H (\overline{\sigma}_\ell)\right)$ (which is irreducible because we know that $\Res _{G_{\Q_q}}^{G_\Q} \left(\Ind^{G_\Q}_H (\overline{\sigma}_\ell)\right) = \Ind^{G_{\Q_q}}_{G_{\Q_{q^4}}} (\overline{\chi}_q) \otimes \overline{\alpha}$) we have that 
\[
\Ind^{G_{\Q_q}}_{G_{\Q_q}  \cap H} \left(\Res^H_{G_{\Q_q} \cap H} (\overline{\sigma}_\ell)\right) = \Ind^{G_{\Q_q}}_{G_{\Q_{q^4}}} (\overline{\chi}_q) \otimes \overline{\alpha}.
\]
Then, from Proposition 3.5 of \cite{ADW14}, it follows that $G_{\Q_{q^4}} \leq  G_{\Q_q} \cap H  = \Gal(\overline{\Q}_q / L_{\mathfrak{q}})$, where $\mathfrak{q}$ is a prime of $L$ above $q$. Thus
\[
\Q_q \subseteq L_{\mathfrak{q}} \subseteq \Q_{q^4} \subseteq \overline{\Q}_q,
\]  
and hence $L$ cannot ramify at $q$ since $\Q_{q^4}$ is an unramified extension of $\Q_q$. 
Moreover, note that
\[
4= \dim (\overline{\rho}_\ell) = \dim \left(\Ind ^{G_\Q} _H (\overline{\sigma}_\ell)\right) = (G_\Q: H) \dim (\overline{\sigma}_\ell)
\]
and 
\[
4 = \dim \left( \Ind^{G_{\Q_q}}_{G_{\Q_q}  \cap H} \left( \Res^H_{G_{\Q_q} \cap H} (\overline{\sigma}_\ell)\right) \right) = (G_{\Q_q}: G_{\Q_q} \cap H ) \dim (\overline{\sigma}_\ell),  
\]
hence $[ L_\mathfrak{q} : \Q_q] = (G_{\Q_q}: G_{\Q_q} \cap H ) = (G_\Q: H) = [L: \Q]$. Therefore, $q$ is inert in $L/\Q$. 

On the other hand, as  $\overline{\rho}_\lambda$ is regular with tame inertia weights at most $k$ and $\ell$ is greater than $24k+1$, by Proposition $3.4$ of \cite{ADW14}, $L$ cannot ramify at $\ell$.
Then, $L \subseteq \Q(i, \sqrt{p_1}, \ldots, \sqrt{p_m})$ and therefore, by assumption, $q$ is split in $L$. Thus we have a contradiction.


\subsection{Symmetric cube case}
 
In order to deal with the case $v)$ of Theorem \ref{clasific} we will use the well-known Dickson's classification of maximal subgroups of $\PGL(2,\F_{\ell^r})$ which states that they can be either a group of upper triangular matrices, a dihedral group $D_{2n}$ (for some integer $n$ not divisible by $\ell$), $\PSL(2, \F_{\ell^s})$, $\PGL(2, \F_{\ell^s})$ (for some integer $s$ dividing $r$), $A_4$, $S_4$ or $A_5$.

Let $G_q$ be the projective image of $\Ind^{G_{\Q_q}}_{G_{\Q_{q^4}}} (\overline{\chi}_q)$. If $G_q$ is contained in a group of upper triangular matrices, this is contained in fact in the subset of diagonal matrices because $\ell$ and $2p$ are coprime. But we know that $G_q$ is non-abelian, then this cannot be contained  in a group of upper triangular matrices.
Moreover, $G_q$ cannot be contained in $A_4$, $S_4$ or $A_5$ because we have chosen $p>13$.

Now, assume that $G_q$ is contained in a dihedral group.  As any subgroup of a dihedral group is either cyclic or dihedral, $G_q$ must be isomorphic to a cyclic or dihedral group. In fact, as $G_q$ is non-abelian, without loss of generality, we can assume that $G_q$ is isomorphic to a dihedral group $D_{2n}$. By definition of dihedral groups we know that a dihedral group of order $2n$ has an element of order $n$. On the other hand, we know that the order of $G_q$ is $4p$. Then, from the equality $4p =2n$, we have that $G_q \cong D_{2n}$ contain an element of order $n=2p$. But, by definition of $G_q$, its elements have order at most $p$, so we have a contradiction. Thus, $G_q$ cannot be contained in a dihedral group.

Hence $G_q$ should be $\PSL(2,\F_{\ell^s})$ or $\PGL(2,\F_{\ell^s})$ for some integer $s$. We know that the stabilizer of a twisted cubic can only occur when $\ell \geq 5$ in which case $\PSL(2,\F_{\ell^s})$ is a index 2 simple subgroup of $\PGL(2, \F_{\ell^s})$. But $G_q$ contains a normal subgroup (of order $p$) of index greater than 2. 
Therefore the case $v)$ in the classification of maximal subgroups of $\GSp(4,\F_{\ell^r})$ cannot occur.  


\subsection{The rest of exceptional cases}

Finally, the order of the groups in case $iv)$ of Theorem \ref{clasific} are $520$, $1440$, $1920$, $3840$ and $5040$. Then all these groups can be discarded by using the fact that the image of $\overline{\rho}^{proj}_{\ell}$ contains an element of order $p>13$.

\end{proof}


\section{Galois representations with large image}

The goal of this section is to prove a representation-theoretic result which gives us the local conditions needed to construct compatible systems without exceptional primes. Roughly speaking, the idea is to construct compatible systems which are maximally induced at two primes simultaneously. 
In order to do this, we start explaining how to choose such primes.

\begin{lemm}\label{pri2}
Let $k,N \in \N$ such that $281 \nmid N$, and $M$ be an integer greater than $N$ and $24k+1$. Let $p'=281$ and $p \equiv 1 \mod 4$ be a prime different from $p'$ and greater than $\max\{M,13 \}$. Then, we can choose two primes $q$ and $q'$ different from $p$ and $p'$ such that: 
\begin{enumerate}
\item $q$ and $q'$ are greater than $M$. 
\item $q'$ is a quadratic residue modulo $q$.
\item $q^2 \equiv -1 \mod p$ and $q'^2 \equiv -1 \mod p'$.
\item $q$ splits completely in $\Q(i,\sqrt{p_1}, \ldots , \sqrt{p_m})$, where $p_1, \ldots, p_m$ are the primes smaller than or equal to $M$.
\item $q'$ splits completely in $\Q(i,\sqrt{p'_1}, \ldots , \sqrt{p'_{m'}})$, where $p'_1, \ldots, p'_{m'}$ are the primes different from $p'$ and smaller than or equal to $M$.
\end{enumerate}
\end{lemm}
 
\begin{proof}
The result follows from Chevotarev's density theorem because $\Q(\zeta_{p})$, $\Q(\zeta_{p'})$, $\Q(\sqrt{q})$, and $\Q(i,\sqrt{p'_1}, \ldots , \sqrt{p'_{m'}})$ are all linearly disjoint over $\Q$.
\end{proof}

The proof of the main result in this section, as in the previous results, relies on the classification of maximal subgroups of $\GSp(4)$. Then we need to know such classification in even characteristic too. 
In this case $\PSp(4,\F_{2^r})=\PGSp(4,\F_{2^r})$, and the maximal subgroups of $\Sp(4,\F_{2^r})$, $r>1$, are as follows (see Section 7.2 and Table 8.14 of \cite{BHRD13}):
\begin{enumerate}
\item the stabilizer of a totally singular or a non-singular subspace;
\item the stabilizer of a decomposition $\F_{2^r}^4 = V_1 \oplus V_2$, $\dim(V_i) = 2$;
\item the stabilizer of a structure of $\F_{2^{2r}}$-vector space on $\F^4_{2^r}$; 
\item $\SO^{+}(4,\F_{2^r})$, $\SO^{-}(4,\F_{2^r})$;
\item the Suzuki group $\Sz(\F_{2^r})$ (when $r$ is odd); 
\item $\Sp(4,\F_{2^{s}})$ for some integer $s>0$ dividing $r$.
\end{enumerate}

\begin{theo}\label{bim}
Let $k$, $N$, $M$, $p$, $p'$, $q$ and $q'$ as in Lemma \ref{pri2}. Consider a compatible system of Galois representations $\rho_\ell: G_\Q \rightarrow \GSp(4,\overline{\Q}_\ell)$ such that, for every prime $\ell$, $\rho_\ell$ ramifies only at the primes dividing $Nqq'\ell$. Assume that for every $\ell >k+2$, $\ell \nmid Nqq'$, a twist of $\overline{\rho}_\ell$ by some power of the cyclotomic character is regular in the sense of Definition 3.2 of \cite{ADW14} with tame inertia weights at most $k$. If $\rho_\ell$ is maximally induced at $q$ of order $p$ (for every $\ell \neq q$) and maximally induced at $q'$ of order $p'$ (for every $\ell \neq q'$), then the image of $\overline{\rho}_\ell^{proj}$ is $\PSp(4,\F_{\ell^s})$ or $\PGSp(4,\F_{\ell^s})$ for all prime $\ell$.
\end{theo}

\begin{proof} Mixing Theorem \ref{clasific} and characteristic 2 classification of maximal subgroups of $\GSp(4)$ we have the following cases.


\subsection{Reducible cases}

As we saw in the proof of Theorem \ref{rib2}, the maximally induced behavior implies that $\overline{\rho}_\ell$ is absolutely irreducible. Indeed, if $\ell \notin \{p,q\}$ then $\overline{\rho}_\ell \vert _{D_q}$ is absolutely irreducible and if $\ell \in \{ p,q \}$, then $\overline{\rho}_\ell \vert _{D_{q'}}$ is absolutely irreducible. Hence, the reducible cases of both classifications cannot occur. 


\subsection{Induced cases}

Now suppose that the image of $\overline{\rho}_\ell$ corresponds to an irreducible subgroup inside
some of the subgroups in cases $ii)$ and $iii)$ of Theorem \ref{clasific}; or in the cases ii) and iii) of characteristic 2 classification.
In these cases, there exist a proper open subgroup $H \subset G_\Q$ of index $2$ and a representation $\overline{\sigma}_\ell : H \rightarrow \GL(2,\overline{\F}_\ell)$ such that $\overline{\rho}_\ell \cong \Ind _H ^{G_\Q} \overline{\sigma}_\ell$. Let $L$ be the quadratic field such that $H = \Gal(\overline{\Q}/ L)$.  Note that, as $ \overline{\rho}_\ell (I_q)$ (resp. $\overline{\rho}_\ell(I_{q'})$) has order $p$ (resp. $p'$) and $\Gal(L/ \Q)$ has order $2$, $L /\Q$ is unramified at $q$ and $q'$.  In fact, as in subsection \ref{indus} (by using the strategy of Lemma 3.7 of \cite{ADW14}),
it can be proved that $q$ and $q'$ are inert. 

Now, if $\ell \notin \{p,q,q'\}$ is greater than $M$, we have that $q$ split completely in $\Q(i,\sqrt{r_1}, \ldots, \sqrt{r_{e}})$ (where $r_1, \ldots , r_e$ are the prime divisors of $N$), $q^2 \equiv -1 \mod p$, $\ell > 24k+1$ and $\ell \nmid Nq'$ then we can apply the same arguments as in subsection \ref{indus} to obtain a contradiction. 
Similarly, if $\ell \notin \{ p',q,q' \}$ we can also apply the arguments of Theorem \ref{rib2}. Thus, we can assume that $\ell \in \{ q,q', p_1, \ldots, p_m \}$, where $p_1, \ldots p_m$ are the primes smaller than or equal to the bound $M$.

Let $\ell \in \{q', p_1, \ldots, p_m \}$.  We know, from the ramification of $\overline{\rho}_\ell$ and from the fact that $L$ is unramified at $q$, that $L$ is contained in $\Q(i, \sqrt{q'}, \sqrt{p_1}, \ldots, \sqrt{p_m})$. Then, by the choice of $q$, it is completely split in $L$. Thus, we have a contradiction.
Finally, from the quadratic reciprocity law we have $\left( \frac{q}{q'}\right) = 1$, then exchanging the roles $q \leftrightarrow q'$ and $p \leftrightarrow p'$ we deal with the case $\ell = q$.


\subsection{Orthogonal cases}

Note that $\SO^+(4, \F_{2^r})$ (resp. $\SO^-(4, \F_{2^r})$) contains a normal subgroup $\Gamma$ of index 2 which is isomorphic to $\PSL(2,\F_{2^r}) \times \PSL(2, \F_{2^r})$ (resp. $\PSL(2,\F_{2^{2r}})$).
Assume that the image of $\Ind^{G_{\Q_q}}_{G_{\Q_{q^4}}} \overline{\chi}_q$ is contained in $\SO^+(4, \F_{2^r})$ or $\SO^-(4, \F_{2^r})$.  Let $L$ the quadratic extension of $\Q$ corresponding to Im$(\overline{\rho}_\ell) \cap \Gamma$ which is contained in $\Q(i, \sqrt{q}, \sqrt{q'}, \sqrt{p_1}, \ldots, \sqrt{p_m})$. Since $\Ind^{G_{\Q_q}}_{G_{\Q_{q^4}}} \overline{\chi}_q$ restricted to $I_q$ is of order $p>2$ it follows that $L$ is unramified at $q$. Then $L$ is contained in $\Q(i, \sqrt{q'}, \sqrt{p_1}, \ldots, \sqrt{p_m})$ which implies that $q$ splits in $L$ and the image of $\Ind^{G_{\Q_q}}_{G_{\Q_{q^4}}} \overline{\chi}_q$ is therefore contained in $\Gamma$.  

If $\Gamma \cong \PSL(2,\F_{2^{2r}})$ we obtain by using the Dickson's classification of maximal subgroups of $\PSL(2,\F_{2^r})$ that the image of $\Ind^{G_{\Q_q}}_{G_{\Q_{q^4}}} \overline{\chi}_q$ cannot be contained in $\Gamma$. Indeed, the image of $\Ind^{G_{\Q_q}}_{G_{\Q_{q^4}}} \overline{\chi}_q$ cannot be contained in a dihedral group $D_{2n}$ because in characteristic 2 we know that $n=(2^r \pm 1)$. Moreover, in such characteristic, the groups $A_4$, $S_4$ and $A_5$ cannot occur.  

The case of groups of upper triangular matrices can be excluded by observing that such groups are isomorphic to the semidirect product of an elementary abelian $2$-group and a cyclic group of order $2^r -1$ and that the image of $\Ind^{G_{\Q_q}}_{G_{\Q_{q^4}}} \overline{\chi}_q$ contain an element of order $4$.

Therefore the image of $\Ind^{G_{\Q_q}}_{G_{\Q_{q^4}}} \overline{\chi}_q$ should be $\PSL(2,\F_{2^s})$ for some integer $s$. As we have chosen $p > 13$ (then $s>1$) we have that $\PSL(2,\F_{2^s})$ is a simple group. But the image of $\Ind^{G_{\Q_q}}_{G_{\Q_{q^4}}} \overline{\chi}_q$ contains a proper normal subgroup of order $p$. Then the image of $\Ind^{G_{\Q_q}}_{G_{\Q_{q^4}}} \overline{\chi}_q$ cannot be contained in $\SO^-(4, \F_{2^r})$.

Finally if $\Gamma \cong \PSL(2,\F_{2^r}) \times \PSL(2,\F_{2^r})$ we have, from the fact that $\Ind^{G_{\Q_q}}_{G_{\Q_{q^4}}} \overline{\chi}_q$ is irreducible, that the image of $\Ind^{G_{\Q_q}}_{G_{\Q_{q^4}}} \overline{\chi}_q$ cannot be contained in $\SO^+(4, \F_{2^r})$ either. Then, the case iv) of characteristic 2 classification cannot occur.


\subsection{Suzuki groups case}

In order to deal with case v) of characteristic 2 classification we have to prove the following result.

\begin{lemm}  
The order of any Suzuki group is not divisible by 281.
\end{lemm}
\begin{proof}
Let $r$ be a positive integer and $\Sz(\F_{2^r})$ be a Suzuki group. We know that the order of $\Sz(\F_{2^r})$ is equal to $2^{2r}(2^{2r}+1)(2^r-1)$ and that the Suzuki group only exist if $r$ is odd.
Suppose that $281$ divides the order of $\Sz(\F_{2^r})$, in particular $281$ divides $(2^{2r}+1)(2^r-1)$. If $281$ divides $(2^r-1)$, then $2^r \equiv 1 \mod 281$. But the order of $2$ modulo $281$ is $70$ then we have a contradiction because $r$ is odd. Then we can assume that $281$ divides $(2^{2r}+1)$, in particular we have that $2^{2r} \equiv -1 \mod 281$ and $2^{4r} \equiv 1 \mod 281$. From this, we have that $70$ divides $4r$ and therefore that $70$ divides $2r$. Thus $2^{2r} \equiv 1 \mod 281$ which is a contradiction too (it contradicts the previous line). 
\end{proof}

By the choice of $p'$ and the previous Lemma, we have that the Suzuki groups cannot occur.


\subsection{Stabilizer of a twisted cubic case}
 
The case $v)$ of Theorem \ref{clasific} was dealt with for all $\ell \notin \{2,p,q\}$ in the proof of Theorem $\ref{rib2}$. Moreover, exchanging the roles $q \leftrightarrow q'$ and $p \leftrightarrow p'$ we deal with the case $\ell \in \{ p,q\}$. Finally, we know that the stabilizer of a twisted cubic does not appear in the classification of maximal subgroups if $\ell < 5$.  


\subsection{The rest of exceptional cases}

The cases $iv)$ of Theorem \ref{clasific} cannot happen because we have chosen $p$ and $p'$ greater than $13$. We remark that by the same reason we exclude the case when $r=1$ in the characteristic 2 classification since the order of $\Sp_4(\F_2)$ is $2^4\cdot 3^2 \cdot 5$.
\end{proof}

\begin{rema}\label{pesos}
Let $\rho_\Pi = (\rho_{\Pi,\lambda})_\lambda$ be a symplectic compatible system of $4$-dimensional semisimple Galois representations attached to a RAESDC automorphic representation $\Pi$ of $\GL(4,\A)$. By part $iii)$ of Theorem \ref{raesdc} this compatible system is Hodge-Tate regular with constant Hodge-Tate weights and for every $\ell \notin S$ and $\lambda \vert \ell$ the representation $\rho_{\Pi,\lambda}$ is crystalline.
Let $a \in \Z$ be the smallest Hodge-Tate weight, $k$ be the biggest difference between any two Hodge-Tate numbers and $\ell \notin S$ be a prime such that $\ell>k+2$. By Fontaine-Lafaille theory the representation $\chi_\ell^a \otimes \overline{\rho}_{\Pi,\lambda}$, $\lambda \vert \ell$, is regular in the sense of Definition 3.2 of \cite{ADW14} with tame inertia weights at most $k$ and the tame inertia weights of this representation are bounded by $k$.
\end{rema}


\section{Automorphic representations without exceptional primes}

In this section we will construct a cuspidal automorphic representation $\Pi$ of $\GL(4, \A)$ such that its associated compatible system satisfies the conditions of Theorem \ref{bim}. Then these compatible system will have "large" image for all primes. 

More precisely, fixing a rational prime $t$ different from $281$ and given $p$, $p'$, $q$ and $q'$ different primes as in Lemma \ref{pri2} (where $N=t$ and $k > 12$), we will construct a cuspidal automorphic representation $\Pi$ of $\GL(4,\A)$ with the following properties:
\begin{enumerate}
\item $\Pi$ is unramified outside $\{t,q,q'\}$,
\item $\Pi \simeq \Pi^\vee$ and the central character of $\Pi$ is trivial,
\item $\rec(\Pi_q) \simeq \WD(\rho_q)$ and $\rec(\Pi_{q'}) \simeq \WD(\rho_{q'})$ (where $\rho_q$ and $\rho_{q'}$ are representations as in Section 4), and
\item $\Pi_\infty$ is of symplectic type and such that $\rec_{\R}(\Pi_\infty)$ is a direct sum of 2-dimensional representations $\sigma_1$ and $\sigma_2$ such that, the restriction of $\sigma_i$ to $\C^{\times}$, $i=1,2$, is of the form $(z/\overline{z})^{\frac{1-2\kappa_i}{2}} \oplus (z/\overline{z})^{- \frac{1-2\kappa_i}{2}}$, where $\kappa_i$ are positive integers such that $\kappa_2 \geq 2$ and $\kappa_1- \kappa_2 >4$.
\end{enumerate} 

\begin{rema}
The conditions in the $\kappa_i$ are in order to assure that $\Pi_\infty$ is a local lift of an integrable discrete series representation $\tau_\infty$ of $\SO(3,2)$ (see Section 5.1 of \cite{KLS}). In particular, according to Section 1, the Hodge-Tate weights are $\{ 2-\kappa_1, 2- \kappa_2, 1+\kappa_2, 1+\kappa_1 \}$. Then, we have that the biggest difference between any two Hodge-Tate numbers is $2\kappa_1-1$, which is greater than 12. So, we must assume that $k$ (as in Remark \ref{pesos}) is greater than $12$ in order to allow the possibility that the conditions in the $\kappa_i$ are satisfied. 
\end{rema}

In order to construct such automorphic representation, we will start by constructing a globally generic cuspidal automorphic representation $\tau$ of $\SO(5,\A)$. Let $\SO(5)$ be the split special orthogonal group of rank 2 defined over $\Q$. Fix two finite and disjoint sets of places: $D = \{\infty, q, q'\}$ and $S=\{ t \}$ such that $\SO(5)$ is unramified at all primes not in $D \cup S$ (this means that $\SO(5,\Z_p)$ is a hyperspecial maximal compact subgroup of $\SO(5,\Q_p)$). 

First, we need to specify what we want at the local places $q$, $q'$ and $\infty$. For such purpose we use the following result of Jiang and Soudry (Theorem 6.4 of \cite{JS03} and Theorem 2.1 and \cite{JS04}), which is a particular case of Theorem 5.3 of \cite{KLS}.

\begin{theo}\label{jsc}
There exist a bijection between irreducible generic discrete series representations of $\SO(5,\Q_q)$ and irreducible generic representations of $\GL(4,\Q_q)$ with Langlands parameter of the form $\sum \sigma_i$ with $\sigma_i$ irreducible symplectic representations which are pairwise non-isomorphic. 
\end{theo}

From this result we have that there is a generic supercuspidal representation $\tau_q$ of $\SO(5,\Q_q)$ (resp. $\tau_{q'}$ of $\SO(5, \Q_{q'})$) which corresponds to a supercuspidal representations $\Pi_q$ of $\GL(4,\Q_q)$ (resp. $\Pi_{q'}$ of $\GL(4,\Q_{q'})$) such that $\rec(\Pi_q) \simeq \WD(\rho_q)$ (resp. $\rec(\Pi_{q'}) \simeq \WD(\rho_{q'})$).
This correspondence is also know at the archimedean places, see section 5.1 of \cite{KLS}. From this we deduce that there is a generic integrable discrete series representation $\tau_\infty$ on $\SO(5,\R)$ which corresponds to the representation $\Pi_\infty$ fixed above with the desired Langlands parameter. 

The following result is a particular case of Theorem 4.5 of \cite{KLS}, which is proved by using Poincar\'e Series.

\begin{theo}\label{pres}
Let $S$ and $D$ disjoint sets as above. Assume that we are given a generic integrable discrete series representation $\tau_\infty$ of $\SO(5,\R)$ and a generic supercuspidal representation $\tau_q$ of $\SO(5,\Q_q)$ for every $q \in D$. Then there exist a globally generic cuspidal automorphic representation $\tau$ of $\SO(5,\A)$ such that the local component of $\tau$ at $\infty$ (resp. at $q$) is $\tau_\infty$ (resp. $\tau_q$) and $\tau_v$ is unramified for every $v$ outside $D \cup S$.
\end{theo} 

Then, this result implies that there exists a globally generic cuspidal automorphic representation $\tau$ on $\SO(5,\A)$ with trivial central character, unramified outside $\{ t, q,q'\}$ and with our desired local components at $q$, $q'$ and $\infty$.

Finally, combining Theorem 7.1 of \cite{CKPSS04} with Theorem E of \cite{JS04}, we can lift $\tau$ to an irreducible automorphic representation $\Pi$ of $\GL(4,\A)$ with trivial central character, such that 
\begin{enumerate}
\item $\Pi \simeq \Pi^\vee $,
\item $\Pi$ is cuspidal, 
\item $L^S(s,\Pi, \wedge ^2)$ has a simple pole at $s=1$, 
\item $\Pi$ is unramified outside $\{ t,q,q' \}$, 
\item $\rec(\Pi_q) \simeq \WD(\rho_q)$, 
\item$\rec(\Pi_{q'}) \simeq \WD(\rho_{q'})$, and 
\item $\Pi_\infty$ has the regular algebraic parameter described in the beginning of the section. 
\end{enumerate} 

Observe that under these conditions the compatible system attached to $\Pi$, as in Theorem \ref{raesdc}, has symplectic images (see Section 5.2 of \cite{KLS}). Then by Theorem \ref{bim} (see also Remark \ref{pesos}), applied to this compatible system, we have the following result. 

\begin{theo}
There are compatible systems $(\rho_\lambda)_\lambda$ such that the image of $\overline{\rho}^{proj}_\lambda$ is $\PSp(4,\F_{\ell^s})$ or $\PGSp(4,\F_{\ell^s})$ for all prime $\lambda$. 
\end{theo}

\begin{rema}
Note that in particular there is an infinite family of RAESDC automorphic representations $(\Pi_n)_{n \in \N}$ of $\GL(4,\A)$ such that, for a fixed prime $\ell$, the size of the image of $\overline{\rho}^{proj}_{\Pi_n,\lambda_n}$ for $\lambda_n \vert \ell$ is unbounded for running $n$, because we can choose $p$ as large as we please by increasing the bound $M$, so that elements of bigger and bigger orders appear in the inertia images.
\end{rema}

On the other hand, by using a result of Jacquet, Piatetski-Shapiro, and Shalika (see, Theorem 9.1 of \cite{KS02}) there exist a globally generic cuspidal automorphic representation $\pi$ of $\GSp(4,\A)$ with trivial central character, such that $\Pi$ is the functorial lift of $\pi$ in the sense of Section \ref{lasec}. Therefore we have the following result.

\begin{theo}
There are  infinitely many globally generic cuspidal automorphic representations of $\GSp(4,\A)$ without exceptional primes.
\end{theo}

Another method to construct automorphic representations with prescribed local conditions is by assuming the Arthur's work on endoscopic classification of automorphic representation for symplectic groups \cite{Art13} and adapting some results of \cite{Shi12}. This work is still conditional on the stabilization of the twisted trace formula and a few expected technical results in harmonic analysis. However, significant progress in this direction has been made by Moeglin and Waldspurger \cite{Wal14a}, \cite{Wal14b}.

This method is used in \cite{ADSW14} in order to construct $2n$-dimensional symplectic compatible systems $(\rho_\lambda)_\lambda$ such that the image of $\overline{\rho}_\lambda$ contain a subgroup conjugated to  $\Sp(2n,\F_\ell)$ for a density one set of primes. The limitation on the set of primes in loc. cit. is due to the authors need to assume the existence of a transvection in order to control the different possibilities for the images of the Galois representations in the compatible system. However, in dimension 4, we eliminate this problem by using the Aschbacher's classification.


\end{document}